\documentclass[12pt,a4paper, reqno]{article}

\usepackage{indentfirst} 
\usepackage{amsmath, amsthm, amssymb,xfrac, mathrsfs} 
\usepackage{tikz} 
\usetikzlibrary{hobby}
\usetikzlibrary{decorations.pathreplacing}
\usepackage{graphicx}
\usepackage{subcaption}
\usepackage{wrapfig} 
\usepackage{xcolor} 
\usepackage{verbatim} 
\usepackage{enumerate} 

\usepackage{hyperref} 
\hypersetup{
colorlinks=true,
citecolor=black!50!red,
linkcolor=black!50!red,
urlcolor=black!50!red,
linktoc=all
}
\usepackage[all]{hypcap} 

\newcounter{constant}

\newcounter{bigconstant}

\newtheorem{teo}{Theorem}[section]
\newtheorem{prop}[teo]{Proposition}

\newtheorem{conjecture}[teo]{Conjecture}

\numberwithin{equation}{section} 

\theoremstyle{definition}

\newtheorem{ex}[teo]{Example}
\newtheorem{remark}[teo]{Remark}
\newtheorem{question}[teo]{Question}

\newcommand{\PP}{\mathbb{P}}
\newcommand{\EE}{\mathbb{E}}
\newcommand{\RR}{\mathbb{R}}

\newcommand{\NN}{\mathbb{N}}
\newcommand{\ZZ}{\mathbb{Z}}

\newcommand{\charf}[1]{\mathbf{1}_{#1}}

\DeclareMathOperator{\dd}{d}

\DeclareMathOperator{\ber}{Bernoulli}
\DeclareMathOperator{\expo}{Exponential}
\DeclareMathOperator{\unif}{U}

\DeclareMathOperator{\med}{Median}

\begin{document}

\title{Majority dynamics and the median process: connections, convergence and some new conjectures}

\author{Gideon Amir\footnote{Email: \ gidi.amir@gmail.com; \ Bar-Ilan University, 5290002, Ramat Gan, Israel} \and Rangel Baldasso\footnote{Email: \ r.baldasso@math.leidenuniv.nl; \ Bar-Ilan University, 5290002, Ramat Gan, Israel} \and Nissan Beilin\footnote{Email: \ message.nissan.music@gmail.com; \ Bar-Ilan University, 5290002, Ramat Gan, Israel}}

\maketitle

\begin{abstract}
We consider the median dynamics process in general graphs. In this model, each vertex has an independent initial opinion uniformly distributed in the interval $[0,1]$ and, with rate one, updates its opinion to coincide with the median of its neighbors. This process provides a continuous analog of binary majority dynamics. We deduce properties of median dynamics through this connection and raise new conjectures regarding the behavior of majority dynamics on general graphs. We also prove these conjectures on some graphs where majority dynamics has a simple description.
\end{abstract}

\section{Introduction}\label{sec:intro}
~
\par Many times the understanding of a stochastic process is obtained by considering different constructions of it, and exploiting properties of such constructions. Not only this, but some of these constructions are interesting on their own, exhibiting many rich and different behaviors. Here, we examine one such model that arises when considering a special construction of the well-studied majority dynamics and that we call median dynamics. The main goal of this paper is to understand properties of the median dynamics process and its connections to majority dynamics.

\par In \emph{majority dynamics}, each vertex $x$ of a fixed underlying graph $G$ receives an opinion that can be either zero or one. With rate one, the opinion at $x$ is updated to match the majority of the neighbors' opinions. When vertices have even degrees, one must choose a way to break ties. Unless otherwise stated, we deal with ties by letting the vertex keep its own opinion. This is equivalent to adding a self-loop in the graph. Another way to deal with tie-breaking is using a fair coin. We will refer to the latter choice as zero-temperature Glauber dynamics (see Section~\ref{sec:ztgd}).

\par While majority dynamics can be run from any initial configuration, it is often natural to let the entries be independent and identically distributed. For $p \in [0,1]$, let $(\xi^{p}_{t})_{t \geq 0}$ denote the process when the initial distribution of opinions is i.i.d.\ with marginals $\ber(p)$, and denote the distribution of this process by $\PP_{p}$.

\par In this paper we consider another process, called the median process, which was first introduced in~\cite{damron} in order to study majority dynamics on the $3$-regular tree. In the \emph{median process} (sometimes referred to as median dynamics), each vertex holds an opinion in the interval $[0,1]$, with initial opinions taken to be i.i.d.\ $\unif([0,1])$ random variables. Each vertex updates its opinion according to an independent Poisson clock with rate one. When the clock rings, the vertex changes its opinion to the median of its neighbors' opinions. In the case when the degree is even, we consider the vertex's own opinion in the pool, in order for the median to be uniquely defined (alternatively, one may use coin flips to choose between the two median values to define a median-version of the zero-temperature Glauber dynamics).

\par Median dynamics generalizes majority dynamics in the following sense (made precise and proved in Section~\ref{sec:median}): if one colors all vertices with opinions in $[0,p]$ as black and all vertices with opinions in $(p,1]$ as white, one retrieves back majority dynamics between the black and white opinions. Thus it is possible to view median dynamics as a coupling of majority dynamics started with i.i.d.\ $\ber(p)$ initial opinions, for all $p\in [0,1]$. We denote the median dynamics process by $(\eta_{t})_{t \geq 0}$ and the probability measure associated to it by $\PP$.

\par The coupling above between median and majority dynamics offers a dictionary by which one can rewrite many results and conjectures regarding majority dynamics in terms of median dynamics. In many cases, these re-formulations seem very natural and offer new ways to approach the topics. More so, the continuous opinions of median dynamics prompt new conjectures concerning majority dynamics that were not previously raised. A large part of this paper is dedicated to such conjectures and to the study of the connections between median and majority dynamics.

\par In addition, the median process is a very interesting process by itself. Another focus of this paper is developing some basic characteristics of the process. We study the dynamics on both general graphs as well as on some special cases. In these special families, the results are attained by first translating results from majority dynamics in general graphs and then strengthening these results using the geometry of the particular graphs (See Subsection~\ref{subsec:examples}).

\par After introducing the processes and establishing the basic connections between majority and median dynamics (see Proposition~\ref{prop:coupling}), we move on to examine properties of the finite-dimensional distributions of the processes and their evolution in time. A rough initial idea one might have for median dynamics is that opinions get more centralized with time, and that central opinions have an easier time spreading than extreme ones. That is, we expect the dynamics  to concentrate the value of the opinion in each vertex towards $\frac{1}{2}$. Looking at the distribution (on $[0,1]$) of the value at a specific vertex $x$, this can be described in two complimentary ways: first, we expect the  distribution of the opinion at any time $t \geq 0$ to have a density that is non-decreasing in the interval $\left[0,\frac{1}{2}\right]$, reflecting the fact that opinions tend to be closer to $\frac{1}{2}$. Second, when considering how these distributions evolve with time, it is natural to expect that they become more concentrated, meaning that the probability of an opinion being in $[0,p]$ at time $t$, for $p \in \left[0, \frac{1}{2}\right]$, should be non-increasing as a function of time.

\par In Section~\ref{sec:marginals}, we precisely state these ideas as conjectures and prove related partial results regarding the marginal distributions of the model. We are not able to prove that these conjectures hold for general underlying graphs, but, by using the duality between majority and median dynamics, we prove them for graphs where majority dynamics has a simple description in Propositions~\ref{prop:K_N} and~\ref{prop:integer_lattice}. These conjectures exhibit some interesting connections to some special families of Boolean functions.

\par One should note that  though we expect opinions to drift towards $\frac{1}{2}$, the geometry of the graph and the tie-breaking mechanism have a decisive role in determining whether opinions converge to $\frac{1}{2}$. In some examples one can show that the marginal distributions converge to a non-trivial limit distribution as time goes to infinity, which even has full support see e.g. Equation~\eqref{eq:support} for the two-dimensional integer lattice), while in others it is not clear whether such a limit exists.

\par We consider convergence properties of the median process in Section~\ref{sec:convergence}. We say that median dynamics converges almost surely on $x \in V(G)$ if the limit $\lim_{t \to \infty} \eta_{t}(x)$ exists almost surely. Convergence of median dynamics can be deduced from fixation of the corresponding majority dynamics. Majority dynamics with parameter $p \in [0,1]$ is said to fixate for $x \in V(G)$ if $\lim_{t \to \infty}\xi^{p}_{t}(x)$ exists almost surely. In this case, the random variables $\xi^{p}_{t}(x)$ are constant from some random time on. It is not hard to verify that, if majority dynamics fixates for all $p \in [0,1]$, then the corresponding median dynamics converges. It is not clear if the converse holds, but we can prove that convergence of median dynamics implies that the corresponding majority dynamics fixates for all but a countable amount of values of $p \in [0,1]$ (Proposition~\ref{prop:convergence}).

\par We do not expect that convergence of median dynamics implies that majority dynamics fixates for all values of $p \in [0,1]$, but we conjecture this is the case for all $p \neq \frac{1}{2}$, see Conjecture~\ref{conj:convergence}.

\par One can also ask whether median dynamics fixates, meaning that the random variables $\eta_{t}(x)$ not only converge, but are constant after some random time. This cannot be directly deduced from fixation of the corresponding majority dynamics and one needs to rely on special properties of the underlying graph and of the dynamics to prove such a result. Relying on dependent percolation arguments, we are able to establish in Theorem~\ref{t:fixation} fixation when the underlying graph is $\ZZ^{2}$.

\bigskip

\par As mentioned earlier, when a vertex has an even number of neighbors one can introduce a different mechanism for tie-breaking by tossing a fair coin. Majority dynamics with coin tosses is known as \textit{zero-temperature Glauber dynamics} (ZTGD) and is related to the well-known Ising model. It is not hard to see that the tie-breaking mechanism can have a profound effect on the processes, and in many cases, such as on $\ZZ^2$, one expects very different behaviors between (standard) majority dynamics and ZTGD.

\par It is also possible to define a process analogous to median dynamics that gives rise to ZTGD, by modifying the dynamics when a vertex has even degree: when such a vertex rings, the opinions of the neighbors are ordered and, independently and uniformly at random, one of the two middle opinions is chosen as the new state of the vertex. Note that most of the theorems and conjectures discussed above hold also for ZTGD.

\par In Section~\ref{sec:ztgd}, we briefly relate these two models and study some simple properties of this modification of median dynamics. Parts of our motivation for studying the median process came from the following folklore conjecture on ZTGD on $\ZZ^2$: that when performing ZTGD on $\ZZ^2$ starting from i.i.d.\ $\ber(p)$ initial conditions for any $p>\frac{1}{2}$, all opinions will eventually fixate on $1$. (See Section~\ref{sec:ztgd} and Conjecture~\ref{conj:ZTGD} for details and background). The language of median dynamics seems to fit well with ZTGD and allows for the phrasing of both weaker and stronger variants of the conjecture.

\bigskip

\noindent \textbf{Related works.} There are many works that consider majority dynamics and related models (see Mossel and Tamuz~\cite{mt} for a survey on models of opinion dynamics). Regarding convergence, Tamuz and Tessler~\cite{tt} prove that, provided the underlying graph $G$ does not grow very fast (in the sense that $M(G,x)$ defined below in Equation~\eqref{eq:M} is finite for every vertex $x$), almost surely, each vertex eventually fixates and, besides, it changes opinion only a bounded amount of times, which they provide estimates for. In their case, this is not a probabilistic result: fixation holds for any initial condition, provided no two clocks ring at the same time. The case of $\ZZ^d$ was proved earlier by Durrett and Steif studying the threshold voter model \cite{durrett1993fixation}.

\par How fast fixation occurs may depend on the initial density and is not completely understood. It is believed that the probability that a vertex changes opinion after time $t$ decays exponentially with $t$. In this direction, Camia, Newman and Sidoravicius~\cite{cns} provides stretched exponential bounds when $p$ is not close to $\frac{1}{2}$ for the hexagonal lattice.

\par In~\cite{ab}, the first and second authors of the current paper study dynamical site percolation on $\ZZ^{2}$ when sites perform majority dynamics. They consider the percolation threshold as a function of time, i.e., the infimum of the initial densities one can consider in order to obtain an infinite component of ones at time $t$. They prove that this function strictly decreases at time zero and that it is continuous. They also prove that there is no percolation at criticality, for all times $t \geq 0$.

\par The discrete-time analog of majority dynamics, where all vertices are updated at once has also been considered. Moran~\cite{moran} and Ginosar and Holzman~\cite{gh} study this dynamics on bounded degree graphs and prove that, under conditions on the underlying graph, it presents the period-two property, which says each vertex eventually has an orbit of period at most two. The period-two property is not observed in all graphs: it is simple to construct an initial condition in the $d$-regular infinite tree that does not present this behavior. Even so, Benjamini, Chan, O'Donnell, Tamuz and Tan~\cite{bcott} prove that, for unimodular transitive graphs, if the initial distribution of opinions is invariant with respect to the automorphism group of the graph, then the period-two property occurs almost surely.

G{\"a}rtner and Zehmakan~\cite{gartner} study the speed of convergence of synchronous majority dynamics on random regular graphs and prove that, provided there is an unbalance in the initial opinions, consensus is reached in sublogarithmic time. Zehmakan~\cite{zehmakan1} extended this result for Erd\"os-R\'enyi random graphs and expander graphs. Fountoulakis, Kang and Makai \cite{fountoulakis2020resolution} proved that for dense Erd\"os-R\'enyi ($p>c\sqrt{n}$) consensus is reached in $4$ rounds w.h.p., and Tran and Vu~\cite{tv}, studied how much the initial opinions needs to deviated from a balanced one in order to obtain overwhelming majority with large probability.

Another type of voting dynamic, sometimes called \emph{local majority rule}, has been considered by the community. Here, each node selects with rate one a random subset of its neighbors (oftentimes with fixed size) and updates its opinion to match the majority of opinions in this randomly selected subset. In this case, convergence to consensus happens generally at a much faster rate, see Cruise and Ganesh~\cite{cg}. Cooper, Els{\"a}sser, Radzik, Rivera, and Shiraga~\cite{cerrs} study this model on expander graphs and verify that, provided the initial discrepancy in the number of opinions is large enough, consensus is reached in logarithmic time. Abdullah and Draief~\cite{ad} consider the same problem on graphs with given degree sequence and establish upper sublogarithmic bounds on the consensus time, for initial conditions have a large enough bias. The same was established for preferential attachment graphs by Abdullah, Bode, and Fountoulakis~\cite{abf}.

\par Regarding median dynamics, Damron and Sen~\cite{damron} recently introduced the process in the context of the $3$-regular tree $T_{3}$ and study it in order to solve conjectures proposed by Howard~\cite{howard} regarding zero-temperature Glauber dynamics. They prove that median dynamics fixates on $T_{3}$, and use this to conclude that, when considering ZTGD, the root of the tree enters an infinite chain of constant opinion in finite time. They also prove that the probability that the root converges to one is a continuous function of $p \in [0,1]$.

\par ZTGD has also received a lot of attention by the community. Our main focus is on the threshold for convergence to the all one configuration. Given a graph $G$, let $p_{c}(G)$ denote the infimum of the densities $p \in [0,1]$ such that, if one starts ZTGD with i.i.d.\ $\ber(p)$ opinions, the process converges almost surely to the configuration that is constant equal to one on every site.

\par Not very much is known about $p_{c}(G)$. When considering the integer lattice, as a consequence of Arratia~\cite{arratia}, one obtains $p_{c}(\ZZ)=1$, while Fontes, Schonmann and Sidoravicius~\cite{fss} proved that $p_{c}(\ZZ^{d}) \in (0,1)$, for all $d \geq 2$. Morris~\cite{morris} established that $p_{c}(\ZZ^{d}) \to \frac{1}{2}$ as $d$ increases. Even though it is widely believed to be true, the following conjecture still remains open.
\begin{conjecture} For all $d \geq 2$,
\begin{equation}\label{conj:ZTGD}
p_{c}(\ZZ^{d})=\frac{1}{2}.
\end{equation}
\end{conjecture}

\par Particular attention has been given to the bi-dimensional case. Nanda, Newman and Stein~\cite{nns} proved that, for $p=\frac{1}{2}$, almost surely, no vertex fixates. A similar result has not been established for $d \geq 3$.

\par Another family of graphs where the process has been studied in depth are the $d$-regular trees $T_{d}$. Here the situation is slightly different, but still not much is known. Howard~\cite{howard} proved that $p_{c}(T_{3}) > \frac{1}{2}$, while Caputo and Martinelli~\cite{cm} concluded that $p_{c}(T_{d}) \to \frac{1}{2}$ as $d$ grows. Whether or not $p_{c}(T_{d})=\frac{1}{2}$ for $d \geq 4$ is not known.

\bigskip

\noindent \textbf{Structure of the paper.} In Section~\ref{sec:median}, we precisely define median and majority dynamics through their generators and prove how to obtain majority dynamics from median dynamics. We consider the marginal measures of the process in Section~\ref{sec:marginals} and examine convergence and fixation properties in Section~\ref{sec:convergence}. Finally, in Section~\ref{sec:ztgd}, we examine how median dynamics might be applied to deduce properties of zero-temperature Glauber dynamics.

\bigskip

\noindent \textbf{Acknowledgments.} This work was supported by the Israel Science Foundation through grant 575/16 and by the German Israeli Foundation through grant I-1363-304.6/2016. We thank Idan Alter for providing us with the simulations in Section~\ref{sec:ztgd}.

\section{Median dynamics}\label{sec:median}
~
\par Given a locally-finite (possibly infinite) graph $G= \big( V(G), E(G) \big)$, denote by $x \sim y$ the adjacency relation in $G$ and define, for $x \in V(G)$, the collection of neighbors of $x$ as
\begin{equation}
N(x)=\{ y \in V(G): y \sim x\}.
\end{equation}
The degree of a vertex $x \in V(G)$ is defined as the cardinality of $N(x)$, and will be denoted by $\deg(x) = |N(x)|$.

We define median dynamics in $V(G)$ as the Markov process $(\eta_{t})_{t \geq0}$ with state space $[0,1]^{V(G)}$ and generator
\begin{equation}
Lf(\eta)=\sum_{x \in V(G)}\left(f(\eta^{x})-f(\eta)\right),
\end{equation}
where $f$ is any bounded continuous local function and $\eta^{x}$ is obtained as a function of $\eta$ by setting
\begin{equation}\label{eq:median_dynamics_flip}
\eta^{x}(y)=\left\{\begin{array}{cl}
\med\{\eta(z): z \in N(x)\}, & \text{if } y=x;\\
\eta(y), & \text{if } y \neq x;
\end{array}
\right.
\end{equation}
where $\med(A)$ denotes the median of the values of $A$. We use the convention that, if $\deg(x)$ is even, we add $x$ to $N(x)$, so that the median is uniquely defined.

\par In this model, every vertex starts with an independent initial value uniformly chosen in the interval $[0,1]$. Besides, each vertex has an associated independent exponential clock that controls its updates. Whenever the clock rings, the value at such vertex is updated to the median of the neighboring values.

\par If the graph $G$ is finite, this process is simply a continuous-time Markov chain with a finite state space. If one considers infinite graphs, tt is not immediately clear that this process is well defined. To verify this, one needs to establish that, at any given finite time, the number of initial values necessary to determine an opinion is finite with probability one. The next proposition says that this is the case when $G$ is a bounded degree graph.

\begin{prop}
If $\deg(x) \leq d$, for all $x \in V(G)$, then median dynamics is well defined.
\end{prop}

\begin{proof}
In order to verify that median dynamics is well defined, it suffices to see that, for any $t \geq 0$ and $x \in V(G)$, $\eta_{t}(x)$ can, almost surely, be determined by observing only finitely many sites at time zero. It is also enough to verify this statement for some small enough time $\delta>0$. From this time on, one iterates the argument to conclude that $\eta_{k \delta}(x)$, $k \in \NN$, can be almost surely determined by finitely many initial opinions. This clearly extends to all positive times $t \geq 0$ and concludes the proof.

For $z \in V(G)$, let $e_{1}(z)$ denote the first random clock ring at $z$. Fix $x \in V(G)$ and denote by $C_{\delta}(x)$ the set of vertices $y \in V(G)$ such that there exists a path $x=x_{1} \sim x_{2} \sim \dots \sim x_{n}=y$ such that, for all $i \leq n$, $e_{1}(x_{i}) \leq \delta$. The value of $\eta_{\delta}(x)$ is determined by the initial condition of the vertices in $C_{\delta}(x) \cup \partial C_{\delta}(x)$\footnote{For a set $A \subset V(G)$, we denote by $\partial A$ the vertex boundary of $A$, i.e., the set of vertices $y \in V(G) \setminus A$ for which $N(y) \cap A \neq \emptyset$.}. Furthermore, since $G$ is a bounded-degree graph, we obtain
\begin{equation}
|C_{\delta}(x) \cup \partial C_{\delta}(x)| \leq (d+1)|C_{\delta}(x)|,
\end{equation}
hence, it suffices to verify that $C_{\delta}(x)$ is almost surely finite, if $\delta$ is small enough. This is a direct consequence that, on bounded-degree graphs, independent site percolation has a strictly positive critical threshold. Let us briefly recall the proof of this fact.

Consider independent site percolation on $G$ with parameter $p \in [0,1]$ and denote by $C(x)$ the open cluster containing $x$. If the site $x$ is closed, $C(x)$ is empty. Notice that, for all $n \in \NN$,
\begin{equation}
\PP_{p}[C(x) \text{ is infinite}] \leq \PP_{p}\left[\begin{array}{cl} \text{there exists an open path} \\ \text{of size $n$ starting at $x$} \end{array}\right].
\end{equation}
Since the number of paths of size $n$ starting at $x$ is bounded by $d^{n}$ and the probability that any given path is open is $p^{n}$, we obtain, for $p < \frac{1}{d}$,
\begin{equation}
\PP_{p}[C(x) \text{ is infinite}] \leq (dp)^{n} \to 0, \text{ as } n \to \infty,
\end{equation}
concluding the proof.
\end{proof}

\par We take the initial condition $(\eta_{0}(x))_{x \in V(G)}$ to be i.i.d.\ uniform variables in $[0,1]$.

\par In majority dynamics, each site $x \in V(G)$ initially receives an opinion that can be either $0$ or $1$. After an exponentially distributed random time, the vertex pools the opinions in $N(x)$ and chooses the most common one. Once again, we add to $N(x)$ the vertex $x$ if $\deg(x)$ is even.

\par The interest in median dynamics is that it gives a canonical coupling of majority dynamics started from i.i.d.\ $\ber(p)$ initial condition.

\begin{prop}\label{prop:coupling}
Let $(\eta_{t})_{t \geq 0}$ denote median dynamics in $G$. Fix $p \in [0,1]$ and define
\begin{equation}
\xi^{p}_{t}(x)= \charf{[0,p]}(\eta_{t}(x)), \quad x \in V(G).
\end{equation}
The process $(\xi^{p}_{t})_{t \geq 0}$ performs majority dynamics on $G$ with initial condition distributed according to i.i.d.\ $\ber(p)$.
\end{prop}

\begin{proof}
Notice first that $\xi^{p}_{t}$ is equal to median dynamics with initial configuration $\charf{[0,p]}(\eta_{0}(x))$. This is a simple consequence of equality
\begin{equation}
\charf{[0,p]}(\med\{\eta(y): y \in N(x)\}) = \med\{\charf{[0,p]}(\eta(y)): y \in N(x)\},
\end{equation}
that follows since $N(x)$ is always odd (recall we are using the convention that we add $x$ to $N(x)$ if $\deg(x)$ is even).

To conclude, it suffices to verify that median dynamics with initial condition $\charf{[0,p]}(\eta_{0}(x))$ is equal to majority dynamics. If $\eta_{0} \in \{0,1\}^{G}$, then the same holds for $\eta_{t}$, for each $t \geq 0$, since, for every $x \in V(G)$, $\eta_{t}(x) \in \{0,1\}$ is a copy of one of the initial values. Besides, if, for some configuration $\eta$, it holds that $\eta(y) \in \{0,1\}$ for all $y \in N(x)$, then $\med\{\eta(y): y \in N(x)\} \in \{0,1\}$ and coincides with the most common opinion in $\{\eta(y): y \in N(x)\}$. In particular, this proves that median dynamics with initial configuration $\charf{[0,p]}(\eta_{0}(x))$ has the same distribution of majority dynamics with density $p \in [0,1]$.
\end{proof}

\section{Marginal measures}\label{sec:marginals}
~
\par In this section, we examine properties of the marginals of the process. We focus on one-dimensional marginals: for $x \in V(G)$, $t \geq 0$ and a measurable set $A \subseteq [0,1]$, set
\begin{equation}\label{eq:marginal}
\mu_{t}^{x}(A) = \PP[\eta_{t}(x) \in A].
\end{equation}
For $\alpha \in [0,1]$, we write $\mu_{t}^{x}(\alpha) = \mu_{t}^{x}([0,\alpha))$. Moreover, when $G$ is a vertex transitive graph, $\mu_{t}^{x}$ does not depend on $x$ and we omit it from the notation.

\par The expected behavior of the dynamics says that the probability of an opinion being in $[0, \alpha)$ should decrease in time if $\alpha \leq \frac{1}{2}$. We state this as a conjecture.
\begin{conjecture}\label{conj:monotonicity}
For any graph $G$, $x \in V(G)$, and $\alpha \leq \frac{1}{2}$, the function $t \mapsto \mu_{t}^{x}(\alpha)$ is monotone non-increasing.
\end{conjecture}

Proposition \ref{prop:coupling} allows us to state an equivalent form of the conjecture in terms of majority dynamics.
\begin{conjecture}\label{conj:monotonicity_maj}
If $p \leq \frac{1}{2}$, then $t \mapsto \PP[\xi^{p}_{t}(x) = 1]$ is monotone non-increasing.
\end{conjecture}

A partial result towards this conjecture was proved in Mossel, Neeman and  Tamuz~\cite{mnt}:
\begin{prop}\label{prop:monotonicity_time_zero}
For any graph $G$, $x \in V(G)$, and $\alpha \leq \frac{1}{2}$, we have $\mu_{t}^{x}(\alpha) \leq \alpha$, for all $\alpha \leq \frac{1}{2}$ and all $t \geq 0$.
\end{prop}

\par The proof goes by considering $\eta_t(x)$ as a (random) Boolean function of a finite subset of the initial configuration. The choice of function depends on the realization of the clocks. The claim then follows from a general statement on Boolean functions (Lemma $5.1$ of~\cite{mnt}), saying that, for every balanced monotone Boolean function $f$, $\textnormal{P}_p(f=1)\geq p$ for all $p\geq \frac{1}{2}$.

By considering the dynamics, one can get the following strengthening of the above.
\begin{teo}\label{t:domination}
For any $\alpha \in \left[0, \frac{1}{2}\right)$, and $\beta \in [0,1]$ such that $\alpha+\beta \leq 1$, the process $\left(\charf{[\beta,\beta+\alpha)}(\eta_{t}(x))\right)_{x \in V(G)}$ stochastically dominates $\left(\charf{[0,\alpha)}(\eta_{t}(x))\right)_{x \in V(G)}$.
\end{teo}

\begin{remark}
Theorem~\ref{t:domination} implies Proposition~\ref{prop:monotonicity_time_zero} for $\alpha$ of the form $\frac{1}{n}$, with $n \in \NN$. From Theorem~\ref{t:domination}, one deduces that $\mu_{t}^{x}\left(\alpha\right) \leq \mu_{t}^{x}\left(\left[k\alpha, (k+1)\alpha \right)\right)$, for each $k \leq n-1$. This and the fact that $\sum_{k=0}^{n-1} \mu_{t}^{x}\left(\left[k\alpha, (k+1)\alpha \right)\right) =1$ imply the proposition for these values of $\alpha$.
\end{remark}

\begin{proof}[Proof of Theorem~\ref{t:domination}]
We will define a coupling between two median dynamics $\eta$ and $\xi$ in such a way that $\left(\charf{[\beta,\beta+\alpha)}(\eta_{t}(x))\right)_{x \in V(G)} \leq \left(\charf{[0,\alpha)}(\xi_{t}(x))\right)_{x \in V(G)}$ pointwise, for all $t \geq 0$, almost surely. This will imply the result.

Set $\gamma=\min\{ \alpha, \beta\}$ and notice that $\alpha+\beta-\gamma= \max\{\alpha, \beta\}$. Select $(\eta_{0}(x))_{x \in V(G)}$ as $\unif[0,1]$ i.i.d.\ random variables. Given $\eta_{0}$, define
\begin{equation}
\xi_{0}(x)=\left\{\begin{array}{cl}
\eta_{0}(x)+\alpha+\beta-\gamma, & \text{if } \eta_{0}(x) \in [0, \gamma);\\
\eta_{0}(x)-(\alpha+\beta-\gamma), & \text{if } \eta_{0}(x) \in [\alpha+\beta-\gamma, \beta+\alpha);\\
\eta_{0}(x), & \text{otherwise.}
\end{array}
\right.
\end{equation}
Notice that $\xi_{0}$ is obtained by $\eta_{0}$ by mapping the opinions that start in the interval $[0, \alpha)$ to the interval $[\beta, \beta+\alpha)$. Furthermore, since $\eta_{0}(x)$ is uniformly distributed, so is $\xi_{0}(x)$, since it is obtained by applying a piecewise-linear map that always has slope one to $\eta_{0}(x)$.

Consider an independent collection of Poisson processes $(\mathscr{P}_{x})_{x \in V(G)}$ of rate one, and use them to run both processes $\eta$ and $\xi$.

We claim that if $\eta_{t}(x) \in [0,\alpha)$, then $\xi_{t}(x) \in [\beta, \beta+\alpha)$. To verify this, consider the set $C_{t}(x)$ of elements of the form $(y,s) \in V(G) \times \RR$ constructed in the following way
\begin{enumerate}
\item $(x, \tau) \in C_{t}(x)$, where $\tau = \sup \left(\mathscr{P}_{x}\cap[0,t]\right) \cup \{0\}$ is the last time before $t$ when there is a clock ring, or zero if no clock rings in $x$ before time $t$;
\item $(y,s) \in C_{t}(x)$ if $s \in \mathscr{P}_{y} \cup \{0\}$ and there exist $z \in N(y) \cup \{y\}$ and $s \leq u \leq t$ such that $(z,u) \in C_{t}(x)$.
\end{enumerate}
The set $C_{t}(x)$ is exactly the set of ordered pairs of locations and times one needs to observe in order to determine both $\eta_{t}(x)$ and $\xi_{t}(x)$ and is almost surely finite.

Let us verify that if $\eta_{s}(y) \in [0, \alpha)$, for some $(y,s) \in C_{t}(x)$, then $\xi_{s}(y) \in [\beta, \beta+\alpha)$. This follows from construction for points where the time coordinate $s$ is zero. Order the times $0 < s_{1} < \dots < s_{i} < \dots < s_{N}$ such that $(y,s_{i}) \in C_{t}(x)$, for some $y \in V(G)$. We proceed by induction on $i$. Assume the claim holds for $j \leq i$ and consider $(y, s_{i+1}) \in C_{t}(x)$. If $\eta_{s_{i+1}}(y) \in [0, \alpha)$, then at least half of the values of $\{\eta_{s_{i+1}-}(z): z \in N(y)\}$ are in $[0,\alpha)$. By induction hypothesis, at least half of the values in $\{\xi_{s_{i+1}-}(z): z \in N(y)\}$ are in $[\beta,\beta+\alpha)$. In particular, the median of this set also lies in $[\beta, \beta+\alpha)$, concluding the claim.
\end{proof}

\par One can phrase a stronger version of Conjecture \ref{conj:monotonicity} by asking whether, for each $\alpha \in \left[0,\frac{1}{2}\right]$, $\mu_{t}^{x}(\alpha)$ is non-increasing given any realization of the Poisson clocks, and considering only the randomness from the initial configuration. This is not the case, as stated in the next proposition.
\begin{prop}
Given $p \in \left[0,\frac{1}{2}\right)$, there exist a graph $G$ and a sequence of vertices $v_1,v_2,..., v_{n}$ of $G$ such that the following holds. If one considers majority dynamics $(\bar{\xi}_{k})_{k=0}^{n}$ with density $p$ and, at step $k$, updating the opinion of $v_{k}$, then there exists a vertex $x$ such that the probability $p_{k} = \PP\left[\bar{\xi}_{k}(x)=1\right]$ is not monotone in $k$.
\end{prop}

\begin{proof}
Let $G=K_{3,m}$ the complete bipartite graph with vertex sets of size $3$ and $m$, where $m$ will be chosen large depending on the value of $p$. More precisely, the vertex set of $G$ is $V=V_{1} \cup V_{2}$, where\footnote{Here we use the standard notation $[n]=\{1, 2, \dots n\}$, for any $n \in \NN$.} $V_{1} = \{1\} \times [3]$ and $V_{2} = \{2\} \times [m]$, and edge set $E=\{(v_{1}, v_{2}): v_{1} \in V_{1} \text{ and } v_{2} \in V_{2}\}$ (see Figure~\ref{fig:bipartite}).

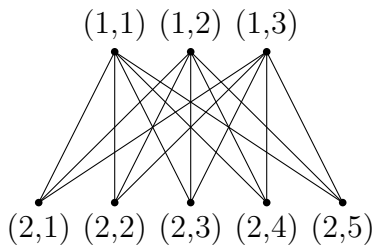
\begin{figure}[h]
\centering
\begin{tikzpicture}

\foreach \x in {1,2,3,4,5}{
\node[below] at (\x-1,0){(2,\x)};
\fill[black] (\x-1,0) circle (0.05);

}

\foreach \y in {1,2,3}{
\node[above] at (\y,2){(1,\y)};
\fill[black] (\y,2) circle (0.05);
\foreach \x in {0,1,2,3,4}{
\draw (\y,2) -- (\x, 0);
}
}

\end{tikzpicture}
\caption{The graph $K_{3,5}$.}
\label{fig:bipartite}
\end{figure}

We define the distinguished sequence of vertices by $v_{1}=(1,1)$, $v_{k}=(2,k-1)$, for $k=2, \dots m+1$, and $v_{m+2} = (1,1)$. Set also $x=(1,1)$. Let us now verify that $p_{k}$ is not monotone.

Clearly $p_{0}=p$. We assume that $m$ is odd, so that $p_{1}$ is the probability that the majority of the vertices on $V_{2}$ are have opinion one. This gives
\begin{equation}
p_{1}=\sum_{j=\frac{m+1}{2}}^{m} \binom{m}{j}p^{j}(1-p)^{m-j}.
\end{equation}
After the first step, we do not update $x=(1,1)$ until step $m+2$, and this implies that $p_{k}=p_{1}$ for $k \in \{1, \dots m+1\}$. Let us now evaluate $p_{m+2}$. Since each vertex in $V_{2}$ is updated once until step $m+1$, all of them will have the same opinion and the probability that this opinion is one at step $m+1$ is $p^{2}+2p_{1}p(1-p)$. At step $m+2$, the distinguished vertex $x=(1,1)$ will copy the opinion of the vertices in $V_{2}$, and we obtain
\begin{equation}
p_{m+2}=p^{2}+2p_{1}p(1-p).
\end{equation}
For any fixed $p \in \left[0, \frac{1}{2}\right)$, the weak law of large numbers implies that $p_{1}$ converges to zero as a function of $m$. In particular, if $m$ is large enough, $p_{1}< p = p_{0}$ and $p_{1} < p^{2} \leq p_{m+2}$, concluding the proof.
\end{proof}

\par One can also expect that, for any time $t \geq 0$, the distribution $\mu_{t}^{x}$ is more concentrated around $\frac{1}{2}$. Let us make this statement more precise. For any $t \geq 0$, $\eta_{t}(x)$ is a copy of one of the initial values $(\eta_{0}(y))_{y \in V(G)}$. In particular, the distribution $\mu_{t}^{x}$ has a density $p_{t}^{x}( \cdot)$. Our second conjecture regards the behavior of $p^{x}_{t}$.
\begin{conjecture}\label{conj:unimodularity}
The density $p_{t}^{x}$ is a unimodular function, meaning that it is symmetric around $\frac{1}{2}$ and monotone non-decreasing in the interval $\left[0,\frac{1}{2}\right]$.
\end{conjecture}

\par It is also possible to rewrite the conjecture above in terms of majority dynamics.
\begin{conjecture}\label{conj:unimodularity_maj}
The function $p \mapsto \PP\left[ \xi^{p}_{t}(x)=1\right]$ is a convex function when $p$ is restricted to the interval $\left[0,\frac{1}{2}\right]$.
\end{conjecture}

\par When one thinks of majority dynamics as a random Boolean function, a possible attempt in proving Conjecture~\ref{conj:unimodularity_maj} is by verifying that, if $f$ is an odd\footnote{A Boolean function $f:\{0,1\}^{n} \to \{0,1\}$ is odd if $f(\omega)=1-f(1-\omega)$, for every $\omega \in \{0,1\}^{n}$.} monotone Boolean function, then $p \mapsto \EE_{p}[f]$ is convex in $\left[0,\frac{1}{2}\right]$. This is not the case, as shown by the following example.

\begin{ex}[Florian Lehner~\cite{Le}]
Consider the function $f: \{0,1\}^{7} \to \{0,1\}$ defined as
\begin{equation}
f(x)=\begin{cases}
1, & \text{if } 5x_{1}+\sum_{i=2}^{7}x_{i}>5, \\
0, & \text{othrewise.}
\end{cases}
\end{equation}
It is easy to observe that $f(x)=1$ if, and only if, $x_{1}=1$ and $x_{i}=1$ for at least one value of $i \geq 2$ or if $x_{1}=0$ and $x_{i}=1$ for all $i \in \{2,3,4,5,6,7\}$. In particular,
\begin{equation}
\EE_{p}[f(x)] = p\left( 1- (1-p)^{6} \right)+(1-p)p^{6},
\end{equation}
which is not a convex function of $p$ for $ p \in \left[0,\frac{1}{2}\right]$.
\end{ex}

\subsection{Examples}\label{subsec:examples}
~
\par Here we prove that Conjectures~\ref{conj:monotonicity} and~\ref{conj:unimodularity} hold for some particular graphs, where the dynamics has a simpler description.

\par Our first example is the complete graph $K_{N}$. In this case, the dynamics reduces to agreeing with the median of the initial opinions after the first ring of each vertex, so it is possible to obtain an explicit expression to $\mu_{t}$, according to whether the clock at a distinguished vertex rang or not.

\begin{prop}\label{prop:K_N}
Conjectures~\ref{conj:monotonicity} and~\ref{conj:unimodularity} are true for the complete graph $K_{N}$.
\end{prop}

\begin{proof}
In order to obtain an explicit expression for $\mu_{t}$, we use the duality between majority and median dynamics provided by Proposition~\ref{prop:coupling}. We declare that a vertex $y$ has opinion 1 at time $t$ if $\eta_{t}(y) \leq \alpha$ and opinion 0 otherwise\footnote{Even though this choice might seem counterintuitive at a first glance, notice that it provides a monotone coupling between different initial conditions: as the value of $\alpha$ increases, more and more sites receive opinion one and this is preserved by the dynamics.}.

Majority dynamics on the complete graph enjoys the very special property that every vertex fixates after its first ring, and it coincides with the majority of the opinions. This fact will be fundamental for us in order to compute $\mu_{t}$, since we can split this probability according to whether the vertex $x$ rings or not before time $t$.

In what follows, we split the discussion in two cases, according to the parity of the size of the complete graph $K_{N}$, $N$.

\textbf{Case 1.} $N$ is odd.

We examine two different possibilities for determining the opinion of $x$ at time $t$. If the vertex does not ring before time $t$, it is necessary that $\eta_{0}(x) \leq \alpha$, which happens with probability $\alpha$. The second possibility is that it rings before time $t$. In this case, the new opinion of the vertex agrees with the majority and hence, for it to be one, it is necessary that most of the vertices have initial opinion one. Setting $N=2n+1$, we can write
\begin{equation}
\mu_{t}(\alpha)=e^{-t}\alpha+(1-e^{-t})\sum_{k=n+1}^{2n+1} \binom{2n+1}{k} \alpha^{k}(1-\alpha)^{2n+1-k}.
\end{equation}

We are now ready to verify the conjectures. We begin by proving that $t \mapsto \mu_{t}(\alpha)$ is non-increasing for all $\alpha \in \left[0,\frac{1}{2}\right]$. We estimate
\begin{equation}
\begin{split}
\frac{\partial}{\partial t}\mu_{t}(\alpha) & = e^{-t}\left[ -\alpha + \sum_{k=n+1}^{2n+1} \binom{2n+1}{k} \alpha^{k}(1-\alpha)^{2n+1-k} \right] \\
& = e^{-t}\left[ -\alpha + (1-\alpha)^{2n+1}\sum_{k=n+1}^{2n+1} \binom{2n+1}{k}\left(\frac{\alpha}{1-\alpha}\right)^{k} \right] \\
& \leq e^{-t}\left[ -\alpha + (1-\alpha)^{2n+1}\left(\frac{\alpha}{1-\alpha}\right)^{n+1}\sum_{k=n+1}^{2n+1} \binom{2n+1}{k} \right] \\
& = e^{-t}\left[ -\alpha + \alpha\left[\alpha(1-\alpha)\right]^{n}2^{2n+1-1} \right] \leq e^{-t}\left[ -\alpha + \alpha\left(\frac{1}{4}\right)^{n}2^{2n} \right] \\
& \leq 0.
\end{split}
\end{equation}

For the second conjecture, we need to consider the derivative of $\mu_{t}$ with respect to $\alpha$ and prove that is non-decreasing for $\alpha \in \left[0, \frac{1}{2}\right]$. This derivative can be written as
\begin{equation}
\frac{\partial}{\partial \alpha}\mu_{t}(\alpha) = e^{-t}+(1-e^{-t})S(\alpha),
\end{equation}
where $S(\alpha)$ is the following summation
\begin{equation}
S(\alpha) = \sum_{k=n+1}^{2n+1} \binom{2n+1}{k} \left[k\alpha^{k-1}(1-\alpha)^{2n+1-k}-(2n+1-k)\alpha^{k}(1-\alpha)^{2n-k} \right].
\end{equation}
We focus on $S(\alpha)$ from now on. We begin by making a change of variables to obtain
\begin{equation}
S(\alpha) = \sum_{k=0}^{n} \binom{2n+1}{k+n+1} \left[(k+n+1)\alpha^{n+k}(1-\alpha)^{n-k}-(n-k)\alpha^{n+k+1}(1-\alpha)^{n-k-1} \right].
\end{equation}
Making use of the following binomial identity
\begin{equation}
\binom{2n+1}{k+n+1} (n-k) = \binom{2n+1}{(k+1)+n+1} ((k+1)+n+1),
\end{equation}
we obtain
\begin{equation}
\begin{split}
S(\alpha) = \sum_{k=0}^{n} & \binom{2n+1}{k+n+1} (k+n+1)\alpha^{n+k}(1-\alpha)^{n-k} \\
& - \binom{2n+1}{(k+1)+n+1} ((k+1)+n+1)\alpha^{n+k+1}(1-\alpha)^{n-k-1},
\end{split}
\end{equation}
and discover that $S(\alpha)$ is a telescopic sum. This yields
\begin{equation}
S(\alpha) = \binom{2n+1}{n+1}(n+1)\alpha^{n}(1-\alpha)^{n},
\end{equation}
which is monotone non-decreasing for $\alpha \in \left[0, \frac{1}{2} \right]$, concluding the first case.

\textbf{Case 2.} $N$ is even.

Once again we split the probability in two cases, depending on whether the vertex has a ring or not before time $t$. When the vertex does not ring before time $t$, the situation is analogous to the previous case, where the opinion is one only if the initial opinion is one. If it rings, the opinion becomes one when most of the initial opinions are one. The case where a difference arises is when exactly half of the opinions are one and half are zero. In this case, with probability $\frac{1}{2}$, the final opinion will be one, and hence the fixed vertex changes to one after the first ring.

If we set $N=2n$, we obtain the expression
\begin{equation}
\mu_{t}(\alpha)=e^{-t}\alpha+(1-e^{-t})\left[\sum_{k=n+1}^{2n} \binom{2n}{k} \alpha^{k}(1-\alpha)^{2n-k}+\frac{1}{2} \binom{2n}{n}\left[\alpha(1-\alpha)\right]^{n}\right].
\end{equation}

With this expression, we are ready to conclude the proof of the two conjectures. Let us start by proving that $t \mapsto \mu_{t}(\alpha)$ is non-increasing for $\alpha \in \left[0,\frac{1}{2}\right]$. Notice that
\begin{equation}
\frac{\partial}{\partial t}\mu_{t}(\alpha) = e^{-t}\left[-\alpha+ \sum_{k=n+1}^{2n} \binom{2n}{k} \alpha^{k}(1-\alpha)^{2n-k}+\frac{1}{2} \binom{2n}{n}\left[\alpha(1-\alpha)\right]^{n}\right].
\end{equation}

Using the fact that $\alpha \leq \frac{1}{2}$, we obtain, for all $n \leq k \leq 2n-1$,
\begin{equation}
\begin{split}
\alpha \binom{2n-1}{k} \alpha^{k}(1-\alpha)^{2n-k-1} & +\alpha \binom{2n-1}{2n-k-2}\alpha^{2n-k-1}(1-\alpha)^{k} \\
& \geq \left[\binom{2n-1}{k} + \binom{2n-1}{k+1}\right]\alpha^{k+1}(1-\alpha)^{2n-(k+1)} \\
& = \binom{2n}{k+1}\alpha^{k+1}(1-\alpha)^{2n-(k+1)}.
\end{split}
\end{equation}
This gives
\begin{equation}
\begin{split}
\sum_{k=n+1}^{2n} & \binom{2n}{k} \alpha^{k}(1-\alpha)^{2n-k}+\frac{1}{2} \binom{2n}{n}\left[\alpha(1-\alpha)\right]^{n} \\
& \leq \alpha \sum_{k=n+1}^{2n-1} \binom{2n-1}{k-1} \alpha^{k-1}(1-\alpha)^{2n-k-2}+\binom{2n-1}{2n-k-3}\alpha^{2n-k-2}(1-\alpha)^{k-1} \\
& \qquad \qquad + \alpha^{2n} + \frac{1}{2} \binom{2n}{n}\left[\alpha(1-\alpha)\right]^{n} \\
& = \alpha \sum_{k=n+1}^{2n-1} \binom{2n-1}{k-1} \alpha^{k-1}(1-\alpha)^{2n-k-2}+\binom{2n-1}{2n-k-3}\alpha^{2n-k-2}(1-\alpha)^{k-1} \\
& \qquad \qquad + \alpha^{2n} + \alpha \binom{2n-1}{n-1}\alpha^{n-1}(1-\alpha)^{n} \\
& = \alpha \sum_{k=0}^{2n-1}\binom{2n-1}{k}\alpha^{k}(1-\alpha)^{2n-1+k} = \alpha.
\end{split}
\end{equation}
In particular, we obtain $\frac{\partial}{\partial t}\mu_{t}(\alpha) \leq 0$ and conclude the proof of Conjecture~\ref{conj:monotonicity}.

As for the second conjecture, we evaluate
\begin{equation}
\frac{\partial}{\partial \alpha}\mu_{t}(\alpha) = e^{-t}+(1-e^{-t})\left[\bar{S}(\alpha)+\frac{n}{2}\binom{2n}{n}\left[\alpha(1-\alpha)\right]^{n-1}(1-2\alpha)\right],
\end{equation}
where $\bar{S}(\alpha)$ is given by
\begin{equation}
\bar{S}(\alpha) = \sum_{k=n+1}^{2n} \binom{2n}{k} \left[k\alpha^{k-1}(1-\alpha)^{2n-k}-(2n-k)\alpha^{k}(1-\alpha)^{2n-k-1} \right].
\end{equation}
The same argument applied in the previous case implies that $\bar{S}(\alpha)$ is a telescopic sum and yields
\begin{equation}
\bar{S}(\alpha) = n\alpha \binom{2n}{n} \left[\alpha(1-\alpha)\right]^{n-1}.
\end{equation}
From this, we obtain
\begin{equation}
\begin{split}
\frac{\partial}{\partial \alpha}\mu_{t}(\alpha) & = e^{-t}+(1-e^{-t})\left[\alpha(1-\alpha)\right]^{n-1}\left[n\alpha \binom{2n}{n} +\frac{n}{2}\binom{2n}{n}(1-2\alpha)\right] \\
& = e^{-t}+(1-e^{-t})\left[\alpha(1-\alpha)\right]^{n-1}\binom{2n}{n}\frac{n}{2},
\end{split}
\end{equation}
concluding the proof that $\frac{\partial}{\partial \alpha}\mu_{t}(\alpha)$ is non-decreasing for $\alpha \in \left[0,\frac{1}{2} \right]$, and Conjecture~\ref{conj:unimodularity} follows.
\end{proof}

\par Next, we consider the case $G=\ZZ$. We use the equivalence between median dynamics and majority dynamics in this case. Majority dynamics on $\ZZ$ is relatively simple: two adjacent vertices with the same opinion are stable, so one might split the graph into finite parts of alternating opinions. Not only this, but fixation at any given vertex occurs after the first ring.
\begin{prop}\label{prop:integer_lattice}
Conjectures~\ref{conj:monotonicity} and~\ref{conj:unimodularity} hold for $G=\ZZ$.
\end{prop}

\begin{proof}
When considering $\ZZ$ as the underlying graph, for any initial configuration $\xi^{p}_{0}$, there exists an unique partition of the integer lattice into intervals of the form $I=[x,y]$ with $\xi^{p}_{0}(x)=\xi^{p}_{0}(x-1)$, $\xi^{p}_{0}(y)=\xi^{p}_{0}(y+1)$ and alternating initial opinions inside $I$ (see Figure~\ref{fig:partition}). These intervals can be of four distinct types, according to the values of $\xi^{p}_{0}(x)$ and $\xi^{p}_{0}(y)$. The importance of this partition is that the dynamics can only evolve inside each of these intervals, and the extremes $x$ and $y$ have constant opinion.

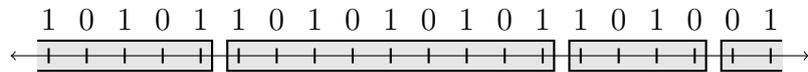
\begin{figure}[h]
\centering
\begin{tikzpicture}

\fill[black!10!white] (-3.15, -0.2) rectangle (-0.85, 0.2);
\fill[black!10!white] (-0.65, -0.2) rectangle (3.65, 0.2);
\fill[black!10!white]  (5.65,-0.2) rectangle (3.85, 0.2);
\fill[black!10!white]  (5.85,-0.2) rectangle (6.65, 0.2);

\draw[ <->] (-3.5,0)--(7,0);
\draw[thick] (-0.65, -0.2) rectangle (3.65, 0.2);
\draw[thick] (-3.15, -0.2) -- (-0.85, -0.2) -- (-0.85, 0.2) -- (-3.15, 0.2);
\draw[thick] (5.65,-0.2) rectangle (3.85, 0.2);
\draw[thick] (6.65, -0.2) -- (5.85, -0.2) -- (5.85, 0.2) -- (6.65, 0.2);

\foreach \x in {-3, -2, -1, -0.5, 0.5, 1.5, 2.5, 3.5, 4, 5, 6.5}{
\node[above] at (\x, 0.2){$1$};
\draw[thick] (\x, -0.1) -- (\x, 0.1);
}

\foreach \x in {-2.5, -1.5, 0, 1, 2, 3, 4.5, 5.5, 6}{
\node[above] at (\x, 0.2){$0$};
\draw[thick] (\x, -0.1) -- (\x, 0.1);
}
\end{tikzpicture}
\caption{The partition of the integer lattice according to the initial condition. Notice that, when the endpoints of an interval have the same opinion, the interval has an odd number of vertices, while the number of vertices is even if the opinions differ at the endpoints.}
\label{fig:partition}
\end{figure}

Our first goal is to obtain an expression for the average number of ones inside a given interval. By translation invariance, we can restrict ourselves to the case $I=[1,k]$, with $k \geq 1$. For $i,j \in \{0,1\}$ and $k \geq 1$, let
\begin{equation}\label{eq:expected_number_ones}
f_{k}^{i,j}(t)=\sum_{n=1}^{k} \EE_{\xi^{i,j,k}}[\xi_{t}(n)],
\end{equation}
where $\xi^{i,j,k}$ is the configuration given by $\xi^{i,j,k}(0)=\xi^{i,j,k}(1)=i$, $\xi^{i,j,k}(k)=\xi^{i,j,k}(k+1)=j$ and alternating opinions inside $[1,k]$. The condition of alternating opinions in $[1,k]$ implies that $k$ is odd when $i=j$ and even if $i \neq j$.

Let us first consider the simpler case when $i \neq j$. Here, $k$ is even and the number of zeros is equal in distribution to the number of ones. This implies
\begin{equation}\label{eq:different_boundaries}
f_{k}^{i,j}(t)=\frac{k}{2}, \quad \text{ for all } t \geq 0.
\end{equation}

We now focus on finding an expression for $f^{1,1}_{k}(t)$. In this case, $k$ is odd and $f^{1,1}_{1}(t)=1$, for all $t \geq 0$. Assume $k \geq 3$. We split the expectation in~\eqref{eq:expected_number_ones} according to the first ring of a vertex inside $[2,k-1]$. If the ring happened after time $t$, then the number of ones at time $t$ is $\frac{k+1}{2}$. When this ring happens before time $t$, we split the expectation according to whether this first ring happened on a vertex with initial opinion one or zero. If the first ring changes an opinion from one to zero, we split the interval $[1,k]$ into two intervals with extremes whose opinions are one and zero and we can obtain the expectation at time $t$ via~\eqref{eq:different_boundaries}. Whenever the first updated site has initial opinion zero, we split the original interval into two intervals whose endpoints have opinion one. The strong Markov property allows us to obtain the following recursive expression.
\begin{equation}
\begin{split}
f_{k}^{1,1}(t)& =\frac{k+1}{2}e^{-(k-2)t} \\
& +\int_{0}^{t}e^{-(k-2)s}\left[\frac{(k-3)(k-1)}{4}+\sum_{n=1}^{\frac{k-1}{2}}1+f^{1,1}_{2n-1}(t-s)+f^{1,1}_{k-2n}(t-s) \right] \, \dd s \\
& = \frac{(k-1)^{2}}{4(k-2)}+\frac{k^{2}-5}{4(k-2)}e^{-(k-2)t}+2e^{-(k-2)t}\int_{0}^{t}e^{(k-2)s}\sum_{n=1}^{\frac{k-1}{2}}f^{1,1}_{2n-1}(s) \, \dd s.
\end{split}
\end{equation}
Combining the expression above with the equality $f^{1,1}_{1}(s)=1$, for all $s \geq 0$, allows us to obtain inductively
\begin{equation}\label{eq:equal_boundaries_one}
f^{1,1}_{k}(t)=\frac{k+3}{2}-e^{-t},
\end{equation}
for all $k \geq 3$ odd.

Finally, we observe that
\begin{equation}\label{eq:equal_boundaries_zero}
f^{0,0}_{k}(t)=k-f^{1,1}_{k}(t)=\frac{k-3}{2}+e^{-t},
\end{equation}
for all $k \geq 3$ odd.

For $i, j \in \{0,1\}$, and $k \geq 1$, let $A(i,j,k)$ be the event that, at time zero, the origin belongs to an interval of the partition that has length $k$ and whose left and right endpoints have respective opinions $i$ and $j$. By translation invariance,
\begin{equation}\label{eq:prob_origin}
\PP\left[\xi^{p}_{t}(0) =1 | A(i,j,k)\right] = \frac{1}{k}f^{i,j}_{k}(t).
\end{equation}
Here, we need to consider triplets $(i,j,k)$ that obey the corresponding parity constrains. Moreover, since there are $k$ possible such intervals and choices of initial opinions containing the origin, we can obtain the following expressions
\begin{equation}\label{eq:prob_intervals}
\PP_{p}[A(i,j,k)] = \left\{\begin{array}{cl}
kp^{3}\left(p(1-p)\right)^{\frac{k-1}{2}}, & \text{if } i=j=1 \text{ and $k$ is odd};\\
k(1-p)^{3}\left(p(1-p)\right)^{\frac{k-1}{2}}, & \text{if } i=j=0 \text{ and $k$ is odd};\\
k\left(p(1-p)\right)^{\frac{k}{2}+1}, & \text{if } i \neq j \text{ and $k$ is even};
\end{array}
\right.
\end{equation}

Combining Equations~\eqref{eq:prob_origin} and~\eqref{eq:prob_intervals} yields
\begin{equation}
\begin{split}
\PP[\xi^{p}_{t}(0)=1] & =p^{3}\sum_{k \text{ odd}}\left(p(1-p)\right)^{\frac{k-1}{2}}f^{1,1}_{k}(t) \\
& \quad + (1-p)^{3}\sum_{k \text{ odd}}\left(p(1-p)\right)^{\frac{k-1}{2}}f^{0,0}_{k}(t) \\
& \quad + 2\sum_{k \text{ even}}\left(p(1-p)\right)^{\frac{k}{2}+1}f^{0,1}_{k}(t).
\end{split}
\end{equation}
Using that the functions $f^{i,j}_{k}$ are given by~\eqref{eq:different_boundaries},~\eqref{eq:equal_boundaries_one} and~\eqref{eq:equal_boundaries_zero}, we can obtain
\begin{equation}
\begin{split}
\PP[\xi^{p}_{t}(0)=1] & =p^{3}+p^{3}\sum_{j=1}^{\infty}\left(p(1-p)\right)^{j}\left(j+2-e^{-t}\right) \\
& \quad +(1-p)^{3}\sum_{j=1}^{\infty}\left(p(1-p)\right)^{j}\left(j-1+e^{-t}\right) \\
& \quad + 2\sum_{j=1}^{\infty}\left(p(1-p)\right)^{j+1}j.
\end{split}
\end{equation}
Evaluating the summations on the expression above and simplifying the polynomials give the expression
\begin{equation}\label{eq:probability_integer_lattice}
\PP[\xi^{p}_{t}(0)=1] = 3p^{2}-2p^{3}-e^{-t}p(1-p)(2p-1).
\end{equation}

To conclude Conjectures~\ref{conj:monotonicity} and~\ref{conj:unimodularity} for $G=\ZZ$, it suffices to see that, for $p \leq \frac{1}{2}$,
\begin{equation}
\frac{\partial}{\partial t}\PP[\xi^{p}_{t}(0)=1] = e^{-t}p(1-p)(2p-1) \leq 0,
\end{equation}
and
\begin{equation}
\frac{\partial^{2}}{\partial p^{2}}\PP[\xi^{p}_{t}(0)=1] = 6(1-e^{-t})(1-2p) \geq 0.
\end{equation}
\end{proof}

\section{Convergence and fixation}\label{sec:convergence}
~
\par Median dynamics and majority dynamics are closely related through Proposition~\ref{prop:coupling}. Hence, asking whether median dynamics converges can be answered by observing fixation in majority dynamics.

\par We say a vertex $x \in V(G)$ fixates for majority dynamics if $\xi^{p}_{t}(x)$ is almost surely constant for $t$ large enough. We say $(\xi^{p}_{t})_{t \geq 0}$ fixates if it fixates for every vertex. The next proposition relates fixation in majority dynamics to convergence in median dynamics.

\begin{prop}\label{prop:convergence}
Median dynamics $(\eta_{t})_{t \geq 0}$ converges almost surely if, and only if, the collection $p \in [0,1]$ such that majority dynamics $(\xi^{p}_{t})_{t \geq 0}$ fixates has full Lebesgue measure.
\end{prop}

\begin{proof}
Assume majority dynamics fixates for almost all $p$ and let $F$ denote the set of values of $p$ for which fixation occurs almost surely. Choose $\tilde{F} \subseteq F$ that is dense and countable. We have
\begin{equation}
\PP\left[ (\xi^{p}_{t})_{t \geq 0} \text{ fixates for all } p \in \tilde{F} \right] =1.
\end{equation}
Fix a realization of $(\eta_{t})_{t \geq 0}$ in the event above and define
\begin{equation}
\eta_{\infty}(x)=\tilde{p}_{x}=\inf\{p \in \tilde{F}: \lim_{t}\xi_{t}^{p}(x)=1\}.
\end{equation}
We claim that $\eta_{t}(x) \to \eta_{\infty}(x)$. In fact, for any $\epsilon>0$, let $p_{-} \in [\tilde{p}_{x}-\epsilon, \tilde{p}_{x}] \cap \tilde{F}$ and $p_{+} \in [\tilde{p}_{x}, \tilde{p}_{x}+\epsilon] \cap \tilde{F}$. We have $\xi^{p_{-}}_{t}(x) \to 0$ and $\xi^{p_{+}}_{t}(x) \to 1$, and this implies
\begin{equation}
\eta_{t}(x) \in [\tilde{p}_{x}-\epsilon, \tilde{p}_{x}+\epsilon], \quad \text{ if $t$ is large enough}.
\end{equation}

Assume now that median dynamics converges almost surely and let $\eta_{\infty}$ denote its limit. We have
\begin{equation}
\PP[\eta_{\infty}(x)=p] \geq \PP[(\xi^{p}_{t}(x))_{t \geq 0} \text{ does not converge}].
\end{equation}
In particular, for any $x$, there exists a countable number of values of $p$ such that $(\xi^{p}_{t}(x))_{t \geq 0}$ does not converge. This concludes the proof.
\end{proof}

In fact, we conjecture that if median dynamics converges, then the corresponding majority dynamics fixates for all $p\neq \frac{1}{2}$.
\begin{conjecture}\label{conj:convergence}
If median dynamics converges, then majority dynamics fixates for any initial density $p \neq \frac{1}{2}$.
\end{conjecture}

\par According to the last proposition, convergence of median dynamics can be understood from fixation in majority dynamics. Conditions on $G$ that assure fixation where considered in Tamuz and Tessler~\cite{tt}. Suppose $G$ is a graph with maximum degree bounded by $d>0$ and, for $x \in V(G)$ and $r >0$, denote by $n_{r}(G,x)$ the number of vertices of $G$ at distance $r$ from $x$. They prove that, if the quantity
\begin{equation}\label{eq:M}
M(G,x)=\sum_{r=1}^{\infty}\left(\frac{d+1}{d-1}\right)^{-r}n_{r}(G,x)
\end{equation}
is finite for every $x \in V(G)$, then majority dynamics fixates for every $p \in [0,1]$. Hence, in graphs that do not grow very fast (meaning that $n_{r}(G,x)$ is a slow growing function for every $x \in V(G)$), median dynamics converges.

\par We already have conditions that imply convergence of median dynamics, but they do not say anything about fixation. If median dynamics fixates, the final configuration has one-dimensional marginals that are absolutely continuous with respect to Lebesgue measure, since any given vertex copies one of the initial opinions. This cannot be directly deduced from convergence, since it might be the case that the process on some vertex converges to some value with positive probability.

\par The next result states that fixation occurs when the underlying graph is the two-dimensional lattice $\ZZ^{2}$.
\begin{teo}\label{t:fixation}
Median dynamics fixates in $\ZZ^{2}$. In particular, the distribution of $\eta_{\infty}(0)$ is absolutely continuous with respect to the Lebesgue measure.
\end{teo}

\begin{proof}
We will verify that $\eta_{t}(0)$ fixates almost surely. The result follows by translation invariance.

Consider the process $(\xi_{t}^{\sfrac{1}{2}}(x))_{t \geq 0}$. Proposition~\ref{prop:coupling} states that it behaves like majority dynamics. We know that $\xi_{\infty}^{\sfrac{1}{2}}(x)=\lim_{t \to \infty} \xi_{t}^{\sfrac{1}{2}}(x)$ exists for all $x \in \ZZ^{2}$. Assume that $\xi_{\infty}^{\sfrac{1}{2}}(0)=1$. If this is not the case, we can use the same argument with the process $(1-\xi_{t}^{\sfrac{1}{2}}(x))_{t \geq 0}$, since both processes $(\xi_{t}^{\sfrac{1}{2}}(x))_{t \geq 0}$ and $(1-\xi_{t}^{\sfrac{1}{2}}(x))_{t \geq 0}$ are equally distributed, once $\PP[\eta_{0}(x)=\sfrac{1}{2}, \text{ for some } x \in \ZZ^{2}]=0$.

As a consequence of the main theorem in~\cite{gandolfi} (see also~\cite[Theorem 3.1]{ab}), one obtains that all connected components of the set $\{x \in \ZZ^{2}: \xi_{\infty}^{\sfrac{1}{2}}(x)=1\}$ are finite. This implies that there exists a $*$-circuit\footnote{A $*$-circuit is a collection of vertices $x_{0}, x_{i} \dots, x_{n}=x_{0}$ such that $||x_{i+1}-x_{i}||_{\infty}=1$, for all $i=0, \dots, n-1$.} $\mathcal{C}$ enclosing the origin such that
\begin{equation}
\xi_{\infty}^{\sfrac{1}{2}}(x)=0, \quad \text{for all } x \in \mathcal{C}.
\end{equation}
Since $\xi_{t}^{\sfrac{1}{2}}(x)$ converges, there exists a random time $t_{0}$ such that, for $t \geq t_{0}$ and $x \in \mathcal{C}$, $\eta_{t}(x)=0$. This gives that, for all $t \geq t_{0}$ and $x \in \mathcal{C}$,
\begin{equation}
\eta_{t}(x)> \sfrac{1}{2}.
\end{equation}
We can take $t_{0}$ large enough so that $\xi_{t}^{\sfrac{1}{2}}(0)=1$, for all $t \geq t_{0}$. This implies that $\eta_{t}(0)$ can only assume the values inside $\text{int}(\mathcal{C})$ for times $t \geq t_{0}$. Since this set is finite, the convergence of $\eta_{t}(0)$ implies fixation.
\end{proof}

\begin{remark}
Fixation is also known to occur for the three-regular tree $T_{3}$. This was proven in~\cite{damron} and also relies on percolative properties of clusters of constant opinion. So far, no proof that does not use this kind of argument is available.
\end{remark}

\par In view of the last result, we conclude that the distribution of $\eta_{\infty}(0)$ is absolutely continuous with respect to Lebesgue measure. One can now ask how does the density of this measure behave. The existence of stable structures for majority dynamics together with the relation in Proposition~\ref{prop:coupling} give some partial results. For example, by considering a $2 \times 2$ square containing the origin, one can obtain the bound
\begin{equation}\label{eq:support}
\PP[\eta_{\infty}(0) \in [\alpha, \beta]] \geq (\beta-\alpha)^{4},
\end{equation}
for all $0 \leq \alpha < \beta \leq 1$.
We conjecture that the main contribution to the density near $0$ comes from stable structures. Since $2\times2$ squares are the smallest stable structures for majority dynamics on $\ZZ^{2}$, we conjecture that the behavior given by the bound above is the correct one.
\begin{conjecture}
For median dynamics on $\ZZ^{2}$,
\begin{equation}
\limsup_{\alpha \to 0} \frac{\PP[\eta_{\infty}(0) \in [0, \alpha]]}{\alpha^{4}} < \infty.
\end{equation}
\end{conjecture}

\section{A note on zero-temperature Glauber dynamics for the Ising model}\label{sec:ztgd}
~
\par Perhaps the most common modification of majority dynamics is zero-temperature Glauber dynamics for the Ising model (ZTGD). While in the former ties are broken by keeping the original opinion, in the latter the new opinion is sampled independently with distribution $\ber\left(\frac{1}{2}\right)$. Of course this does not change anything when the underlying graph has only vertices with odd degree, but it affects the dynamics in a non-trivial way when this is not the case.

\par Let us define the model more precisely. ZTGD is a Markov process $(\xi_{t})_{t \geq 0}$ with state space $\{0,1\}^{G}$ and generator given on local functions $f: \{0,1\}^{G} \to \RR$ by
\begin{equation}
\bar{L}f(\xi)= \sum_{x \in V(G)} \left(\charf{\{\sum_{y \sim x}|\xi(y)-\xi(x)| > \frac{\deg(x)}{2}\}}+\frac{1}{2}\charf{\{\sum_{y \sim x}|\xi(y)-\xi(x)| = \frac{\deg(x)}{2}\}}\right)\left(f(\xi^{x})-f(\xi)\right),
\end{equation}
where $\deg(x)$ is the degree of $x \in V(G)$ and $\xi^{x}$ is the configuration obtained from $\xi$ by changing the configuration only at $x$
\begin{equation}
\xi^{x}(z)=\left\{\begin{array}{cl}
1-\xi(x),& \mbox{if}\,\,\, z=x,\\
\xi(z),& \mbox{otherwise.}
\end{array}
\right.
\end{equation}

\bigskip

\par It is possible to get a version of the median process which generalizes ZTGD. We call this process median dynamics with coin flips (MD$_{coins}$). In this process, each vertex $x \in V(G)$ receives an initial opinion $\bar{\eta}_{0}(x) \in [0,1]$ and an independent clock with $\expo(1)$ distribution. When the clock rings, the opinion of $x$ is updated to the median of the opinions. The median is not uniquely defined if $\deg(x)$ is even, and in this case we flip a fair coin to randomly select one of the two middle opinions. This can be precisely defined as the Markov process $(\bar{\eta}_{t})_{t \geq 0}$ with state space $[0,1]^{G}$ and generator
\begin{equation}
\begin{split}
L^{coins}f(\bar{\eta}) & = \sum_{x \in V(G) \,: \, \deg(x) \text{ is even}}\frac{1}{2}\left(f(\bar{\eta}^{x,+})-f(\bar{\eta})\right) + \frac{1}{2}\left(f(\bar{\eta}^{x,-})-f(\bar{\eta})\right)  \\
& \qquad + \sum_{x \in V(G) \,:\, \deg(x) \text{ is odd}}\left(f(\bar{\eta}^{x})-f(\bar{\eta})\right),
\end{split}
\end{equation}
where $f:[0,1]^{G} \to \RR$ is any bounded continuous local function, $\bar{\eta}^{x}$ is given by~\eqref{eq:median_dynamics_flip}, and $\bar{\eta}^{x,+}$ and $\bar{\eta}^{x,-}$ are configurations obtained from $\bar{\eta}$ by changing only the entry at $x$ to the smaller and larger middle values in $\{\bar{\eta}(y): y \in N(x)\}$, respectively.

\par Intuitively, the relationship between MD$_{coins}$ and ZTGD mimics that between the median process and majority dynamics. All results of Section~\ref{sec:median} as well as Proposition~\ref{prop:convergence} remain valid, and we believe the conjectures on marginal measures and convergence (Conjectures~\ref{conj:monotonicity},~\ref{conj:unimodularity} and~\ref{conj:convergence}) should also hold in this case. In fact, one can check, in similar ways, that Conjectures~\ref{conj:monotonicity} and~\ref{conj:unimodularity} also work for MD$_{coins}$ in the complete graph $K_{N}$, obtaining an analogous to Proposition~\ref{prop:K_N}.

\par As mentioned in the introduction, on some graphs one observes very different behaviors in majority dynamics and in ZTGD. Specifically, on $\ZZ^2$, it is well known that  majority dynamics fixates for any initial condition, while for ZTGD this is known not to hold for i.i.d.\ $\ber\left(\frac{1}{2}\right)$ initial condition. More so, for i.i.d.\ $\ber(p)$ initial condition, with $p \in \left(\frac{1}{2}, 1\right]$, majority dynamics fixates on a non-trivial configuration, while Conjecture \ref{conj:ZTGD} says that in the limit configuration all opinions are one. Part of the reason for this difference is that, for majority dynamics on $\ZZ^2$, there are many finite stable configurations - that is, finite subsets of $\ZZ^2$ that, if they all share the same value, the opinions on that set cannot change anymore (for instance a $2\times 2$ square and actually any cycle). ZTGD on the other hand has no finite stable structures on $\ZZ^2$.

\begin{figure}
  \centering
  \begin{subfigure}{0.3\linewidth}
    \includegraphics[width=\linewidth, trim={5.5cm 9cm 4cm 9cm}, clip]{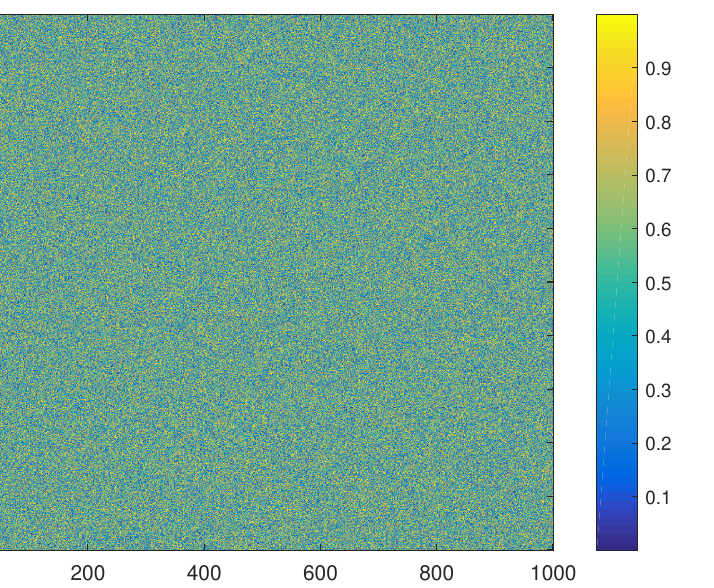}
  \end{subfigure}
  \begin{subfigure}{0.3\linewidth}
    \includegraphics[width=\linewidth, trim={5.5cm 9cm 4cm 9cm}, clip]{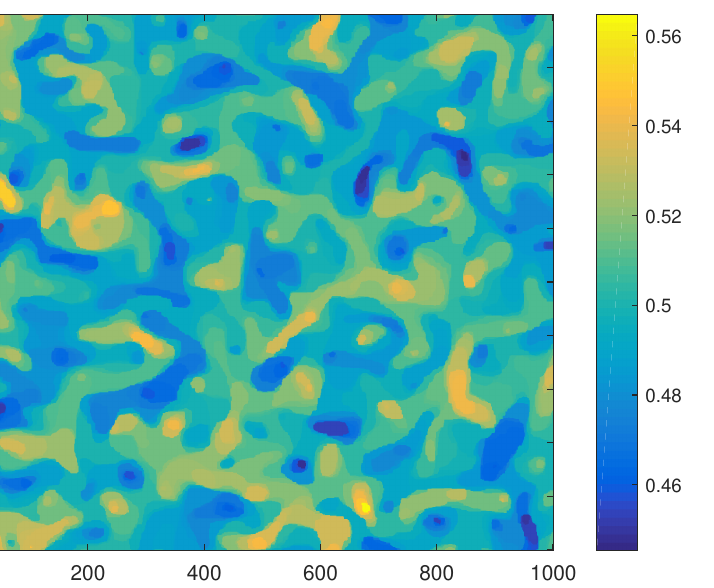}
  \end{subfigure}
  \begin{subfigure}{0.3\linewidth}
    \includegraphics[width=\linewidth, trim={5.5cm 9cm 4cm 9cm}, clip]{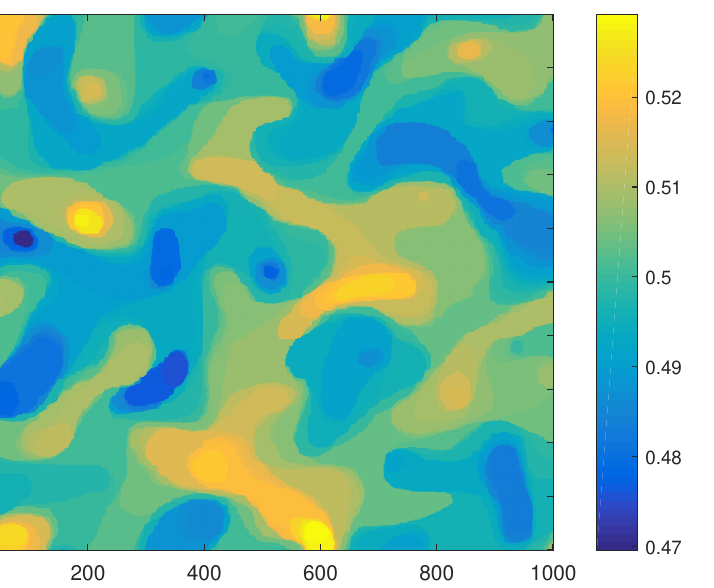}
  \end{subfigure}
  \begin{subfigure}{0.3\linewidth}
    \includegraphics[width=\linewidth, trim={5.5cm 9cm 4cm 9cm}, clip]{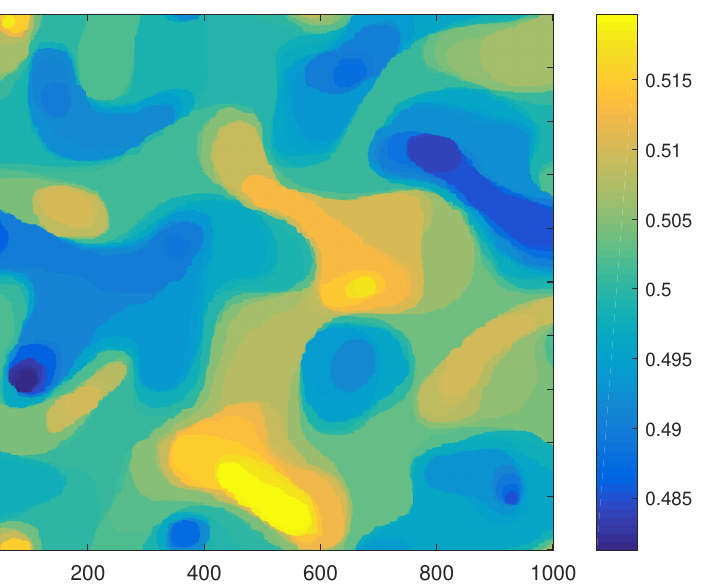}
  \end{subfigure}
  \begin{subfigure}{0.3\linewidth}
    \includegraphics[width=\linewidth, trim={5.5cm 9cm 4cm 9cm}, clip]{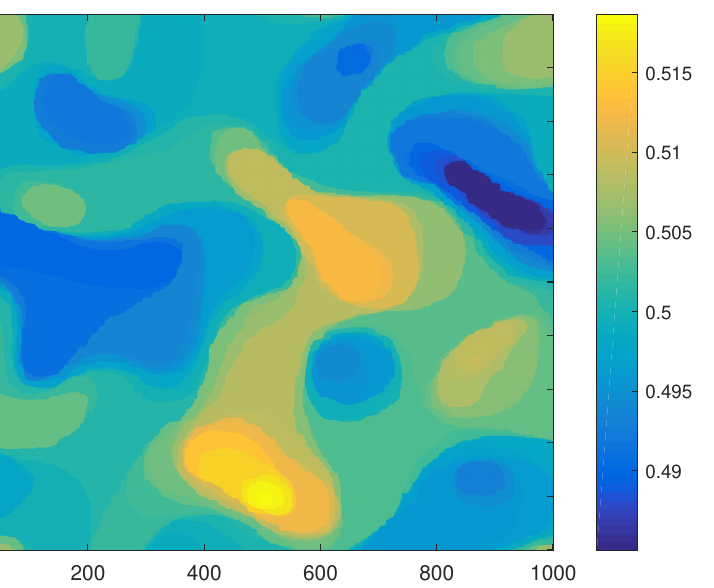}
  \end{subfigure}
  \begin{subfigure}{0.3\linewidth}
    \includegraphics[width=\linewidth, trim={5.5cm 9cm 4cm 9cm}, clip]{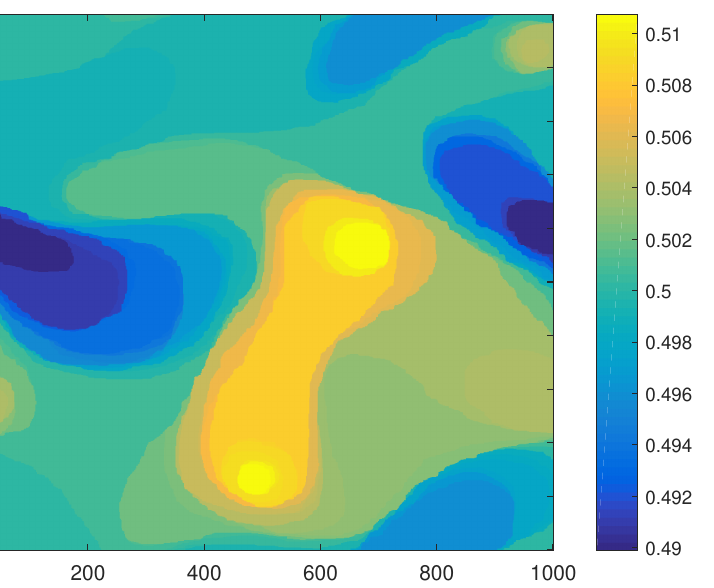}
  \end{subfigure}
  \caption{A simulation of MD$_{coins}$ in $\ZZ^{2}$ for times $t=0, 10^{3}, 5\cdot 10^{3}, 10^{4}, 1.5 \cdot 10^{4}, 3 \cdot 10^{4}$. Notice the colors in each picture change according to the scale located at the right of it.}
  \label{fig:simulation}
\end{figure}

\par Should Conjecture~\ref{conj:ZTGD} hold, it would imply that the corresponding MD$_{coins}$ converges almost surely to $\frac{1}{2}$. This would provide an example where Conjecture~\ref{conj:convergence} holds without convergence for $p=\frac{1}{2}$. On the other hand, if MD$_{coins}$ converges in $\ZZ^{2}$, then the limit needs to be almost surely equal to $\frac{1}{2}$. This is a direct consequence of the fact that $\charf{\left[0,\frac{1}{2}\right]}(\eta_{t})$ is a ZTGD with initial density $\frac{1}{2}$, that almost surely does not converge, as proved in~\cite{nns}.
By the discussion above, we can rephrase Conjecture \ref{conj:ZTGD} in terms of MD$_{coins}$:
\begin{conjecture}[Conjecture \ref{conj:ZTGD}, rephrased]
For every $d\geq 2$, in MD$_{coins}$ on $\ZZ^d$ the opinion at the origin converges to $\frac{1}{2}$ a.s., that is,
\begin{equation}
\eta_t(0)\xrightarrow[t\rightarrow\infty]{a.s.} \frac{1}{2}.
\end{equation}
\end{conjecture}
The richer structure from MD$_{coins}$ allows us to state weaker and stronger versions of this conjectures.

\par One can, for example, weaken Conjecture~\ref{conj:ZTGD} by changing almost sure convergence to a weaker form of convergence, such as convergence in probability.
\begin{conjecture}\label{conj:convergence_probability}
For every $d \geq 2$, in MD$_{coins}$ on $\ZZ^{d}$ the opinion of the origin converges in probability to $\frac{1}{2}$ as time grows, that is,
\begin{equation}
\eta_t(0)\xrightarrow[t\rightarrow\infty]{\text{probability}} \frac{1}{2}.
\end{equation}
\end{conjecture}
In terms of majority dynamics, Conjecture \ref{conj:convergence_probability} means that when starting from i.i.d.\ $\ber(p)$ initial conditions, for any $p>\frac{1}{2}$ the frequency at which we see zeroes at any fixed position goes to $0$.

\par Conjecture~\ref{conj:ZTGD} can be strengthened by asking what is the rate of convergence of MD$_{coins}$. We can consider how the opinion of each site deviates from $\frac{1}{2}$ or the energy per site. For $\eta \in [0,1]^{G}$, let $H$ denote the energy at a distinguished  site $x \in V(G)$
\begin{equation}\label{eq:energy}
H(\eta)=\sum_{y \sim x} |\eta(y)-\eta(x)|.
\end{equation}
The definition of energy is motivated by the fact that, given a collection of real values $\{x_{i}\}_{i=1}^{n}$, their median is a minimizer of the function $y \mapsto \sum_{i=1}^{n}|y-x_{i}|$. This implies that, when $x$ rings, the energy at $x$ cannot increase. (When $n$ is even any value between the two medians is a minimizer, but the energy still cannot increase).

\par If Conjecture~\ref{conj:ZTGD}~holds, then
\begin{equation}
H(\eta_{t}) \to 0 \text{ and } \left|\eta_{t}(0)-\frac{1}{2}\right| \to 0
\end{equation}
on $\ZZ^{2}$. Simulations (see Figure~\ref{fig:simulation}) seem to imply that both quantities above have a power-law decay.

\begin{question}
Do there exist positive constants $\alpha , \beta >0$ such that
\begin{equation}
H(\eta_{t}) \leq t^{-\alpha+o(1)} \text{ and } \left|\eta_{t}(0)-\frac{1}{2}\right| \leq t^{-\beta+o(1)},
\end{equation}
where $o(1)$ denotes a function converging to $0$ as $t\rightarrow \infty$?
\end{question}

\par We are still not able to prove that there is convergence for MD$_{coins}$, but we prove a partial result considering the energy per site. Our next result states that, on $\ZZ^{d}$, the number of times a flip at $x$ can provoke an energy change that is bigger than $\epsilon>0$ in absolute value is almost surely finite.

\begin{prop}
On $\ZZ^{d}$, for any $\epsilon>0$, there are almost surely finitely many flips of the origin with a change in energy of at least $\epsilon$ in absolute value.
\end{prop}

\begin{proof}
We can write the energy at time $t$ as the sum of the contribution of the energy from all clock rings at the origin and its neighbors before time $t$. This gives
\begin{equation}
\begin{split}
H(\eta_{t})-H(\eta_{0}) & = \sum_{\substack{s \leq t: \text{ 0 flips}\\ \text{at time } s}} \Big(H(\eta_{s})-H(\eta_{s-})\Big) \\
& \quad + \sum_{y \sim 0} \sum_{\substack{s \leq t: \text{ $y$ flips}\\ \text{at time } s}} \Big( |\eta_{s}(y)-\eta_{s}(0)|-|\eta_{s-}(y)-\eta_{s-}(0)| \Big).
\end{split}
\end{equation}

By translation invariance on each coordinate axis, we have
\begin{multline}
\EE\left[\sum_{y \sim 0} \sum_{\substack{s \leq t: \text{ $y$ flips}\\ \text{at time } s}} \Big( |\eta_{s}(y)-\eta_{s}(0)|-|\eta_{s-}(y)-\eta_{s-}(0)| \Big)\right] \\ = \EE\left[\sum_{\substack{s \leq t: \text{ 0 flips}\\ \text{at time } s}} \Big(H(\eta_{s})-H(\eta_{s-})\Big) \right],
\end{multline}
and this yields
\begin{equation}
\EE\left[ H(\eta_{t})-H(\eta_{0})\right] = 2\EE\left[\sum_{\substack{s \leq t: \text{ 0 flips}\\ \text{at time } s}} \Big(H(\eta_{s})-H(\eta_{s-})\Big) \right].
\end{equation}

For $\epsilon>0$, if $\mathcal{N}_{\epsilon}(t)$ denotes the number of flips of the origin with energy change of at least $\epsilon$ in absolute value, we obtain
\begin{equation}
-\epsilon \mathcal{N}_{\epsilon}(t) \geq \sum_{\substack{s \leq t: \text{ 0 flips}\\ \text{at time } s}} \Big(H(\eta_{s})-H(\eta_{s-})\Big),
\end{equation}
and this implies
\begin{equation}
\EE\left[ \mathcal{N}_{\epsilon}(t)\right] \leq \frac{1}{2\epsilon} \EE\left[ H(\eta_{0})-H(\eta_{t})\right] \leq \frac{d}{\epsilon}.
\end{equation}
If $\mathcal{N}_{\epsilon}=\lim_{t}\mathcal{N}_{\epsilon}(t)$ denotes the number of flips whose energy change is at least $\epsilon$ in absolute value, we can use the monotone convergence theorem to conclude that $\EE[\mathcal{N}_{\epsilon}]$ is finite. In particular, $\mathcal{N}_{\epsilon}$ is almost surely finite. This concludes the proof.
\end{proof}

\begin{remark}
The proof of the proposition above also works for general Cayley graphs, since it only relies on the mass-transport principle.
\end{remark}

\bibliographystyle{plain}
\bibliography{mybib}

\end{document}